\documentclass[a4paper,12pt]{amsart}
\usepackage{amsmath,amsthm,amssymb}
\usepackage{hyperref}

\allowdisplaybreaks[1]

\numberwithin{equation}{section}
\newtheorem{theorem}{Theorem}[section]
\newtheorem{lemma}[theorem]{Lemma}
\newtheorem{proposition}[theorem]{Proposition}
\newtheorem{corollary}[theorem]{Corollary}
\theoremstyle{remark}
\newtheorem{remark}[theorem]{Remark}
\newtheorem{example}{Example}[section]
\newtheorem{definition}{Definition}[section]

\newcommand{\dl}{\delta}
\newcommand{\e}{\epsilon}
\newcommand{\io}{\iota}
\newcommand{\w}{\omega}

\newcommand{\ity}{\infty}
\newcommand{\C}{\mathbb{C}}

\newcommand{\Ob}{\mathcal{O}}
\newcommand{\ti}{\widetilde}

\newcommand{\ta}{\theta}

\newcommand{\N}{\mathbb{N}}

\newcommand{\al}{\alpha}

\begin{document}

\title[attractors and chain recurrence  in generalized semigroup]
{attractors and chain recurrence in generalized semigroup}

\author[K. Lalwani]{Kushal Lalwani}
\address{Kushal Lalwani\\Department of Mathematics\\ University of Delhi\\Delhi--110 007, India}
\email{lalwani.kushal@gmail.com }

\thanks{This work was supported by research fellowship from University Grants Commission (UGC), New Delhi.}


\begin{abstract}
In \cite {kl1}, we  extended various notions of recurrence for the generalized semigroup analogous to their counterpart in the classical  theory of dynamics. In this paper, we shall address the alternative definition of chain recurrent  set in terms of attractors, given by Hurley in \cite {mh} following Conley\rq{}s characterization in \cite {conley}. We shall also discuss the notion of topological transitivity and chain transitivity in this general setting.
\end{abstract}

\maketitle

\section{Introduction}

The main aim  of this paper is to study the notion of attractors and chain recurrence in the context of generalized semigroup. This concept for flows (continuous action of a group) on a compact metric space was originally introduced by Conley \cite {conley}. Later, Hurley in \cite {mh} extended this characterization for semiflows (continuous action of a semigroup) without the assumption of compactness. Here, we shall follow the treatment of Hurley for compact metric spaces and see how far this characterization applies in this setting of generalized semigroup.

A {\it continuous semigroup}  is a set of (non-identity) continuous self maps, of a topological space $X$,  which are closed under the composition. A semigroup $G$ is said to be generated by a family $\{g_{\al}\}_{\al}$ of  continuous self maps of a topological space $X$  if every element of $G$ can be expressed as compositions of iterations of the elements of  $\{g_{\al}\}_{\al}$. We denote this by $G=<g_{\al}>_{\al}$. The space $X$ is assumed to be Hausdorff and first countable.

Given a semigroup $G$ and $x \in X$, the set $\Ob_G(x):=\{g(x): g \in G\}\cup \{x\}$ is called the  {\it orbit} of $x$ under $G$.
A subspace $Y \subset X$ is said to be  {\it{invariant}} under $G$ if $g(y)\in Y$ for all $g\in G$ and $y\in Y$.
If $G=\ <g_{\al}>_{\al}$ then for $Y$ to be invariant it is sufficient that $g(y)\in Y$ for all $g\in \{g_{\al}\}_{\al}$ and $y\in Y$.

In \cite {kl}, we have introduced the notion of topological conjugacy of dynamics of two semigroups on two topological spaces as follows:
Let $G=\ <g_i>_{i \in \varLambda}$ be a semigroup of continuous self maps of $X$ and $\ti{G}=\ <\ti{g_i}>_{i \in \varLambda}$ be a semigroup of continuous self maps of $Y$. Two dynamical systems $(X,G)$ and $(Y,\ti{G})$ are said to be topologically conjugate if there exists a homeomorphism $\rho : X \to Y$ such that $\rho\circ g_i=\ti{g_i}\circ \rho$ for each $i \in \varLambda$.
Note that for $g \in G$, we have, $g=g_{i_1}\circ \ldots \circ g_{i_n}$ for some $g_{i_j}\in \{g_i\}_{i \in \varLambda}$.  If $\rho : X \to Y$ is a topological conjugacy then
\begin{equation} \notag
\begin{split}
\rho \circ g &= \rho \circ g_{i_1}\circ \ldots \circ g_{i_n}\\
&=\ti{g}_{i_1}\circ \ldots \circ \ti{g}_{i_n}\circ \rho\\
&=\ti{g} \circ \rho
\end{split}
\end{equation}
for $\ti{g} \in \ti{G}$.

Hurley had shown that chain recurrent set for a semiflow is the complement of the union of the set $B(A) \setminus A$, as $A$ varies over the collection of attractors and $B(A)$ denotes the basin of attraction. In setting of a generalized semigroup, the situation is fairly complicated. The omega limit set $\w (x)$ for a point in the phase space need not be same as the omega limit set for any of the point in the orbit of $x$, classically these sets are same. Since the concept of attractor is based on  the notion of limit set, we need to reformulate it in general setting. 


Furthermore, we shall discuss the concepts of transitivity in the context of a generalized semigroup. We define topological transitivity and chain transitivity in this more general setting, analogous to their counterpart in the classical theory. Subsequently, we shall develop a parallel theory of transitivity for semigroups.

Although, this is a self-contained exposition, for general reference to standard terms and basic facts from the classical theory of dynamics we recommend  Alongi and Nelson\rq{}s book \cite{Alongi}.

\section{transitivity}

In this section, we define the notion of topological transitivity and chain transitivity for generalized semigroup. The topological transitivity is a global attribute of a dynamical system. A topological transitive system has points which move from one neighborhood to any other, under the dynamics of some map. Thus, the dynamical system can not be decomposed into two invariant subsystems. For a general semigroup we define topological transitivity as follows:

\begin{definition}
The dynamical system $(X,G)$ is {\it topologically transitive} if for any two nonempty open subsets $U$ and $V$  of $X$, there exists a $g \in \widehat G=G\cup \{ identity\}$ such that $g(U) \cap V\neq \emptyset$. A nonempty invariant  subset $A$ of $X$ is said to be {\it topologically transitive}  if for any two nonempty open subsets $U$ and $V$  of $A$, there exists a $g \in \widehat G$ such that $g(U) \cap V\neq \emptyset$.
\end{definition}

\begin{remark} \label {rktt}
Equivalently, the definition for topological transitivity can also be formulated as: if for any two nonempty open subsets $U$ and $V$  of $X$, there exists a $g \in \widehat G=G\cup \{ identity\}$ such that $U \cap g^{-1}(V)\neq \emptyset$. Furthermore, since each $g \in G$  is a continuous self map of $X$, the set $g^{-1}(V)$ is also open in $X$. Therefore, by definition of topological transitivity, we have, given two  open subsets $U$ and $V$ and a $g\in G$ , there exists a $h \in \widehat G$ such that $U \cap h^{-1}\circ g^{-1}(V)\neq \emptyset$, or equivalently, $g\circ h(U) \cap V\neq \emptyset$.
\end{remark}

Also, some authors define the concept of topological transitive system in terms of existence of a dense orbit. We show that if the phase space is second countable Baire space then the topological transitive system has a dense orbit.

\begin{lemma} \label{bt}
Let  $G$ be a semigroup of continuous self maps of a second countable Baire space $X$. If \ $\bigcup_{g \in \widehat G} g^{-1}(U)$ is dense in $X$ for every nonempty open subset $U$ of $X$, then there exists $D\subset X$, such that $D$ is residual in $X$ and $\Ob_G(x)$ is dense in $X$ for all $x \in D$.
\end{lemma}

\begin{proof}
Let $\{U_i\}_{i\in \N}$ be a countable base for $X$. Then, for each $i\in \N$, $\bigcup_{g \in \widehat G} g^{-1}(U_i)$ is dense in $X$. Since each $g \in G$ is continuous and $U_i $  is open, hence $g^{-1}(U_i)$ is open. Thus, $\bigcup_{g \in \widehat G} g^{-1}(U_i)$ is open in $X$ for each $i \in \N$. Therefore, the set 
$$D=\bigcap_{i\in \N}\bigcup_{g \in \widehat G} g^{-1}(U_i)$$
is residual in $X$ and hence dense in $X$.

For $x \in D$, we have, $x \in \bigcup_{g \in \widehat G} g^{-1}(U_i)$ for each $i \in \N$. Then for each $i=1,2,3,\ldots$ there exists a $g_i \in \widehat G$ such that $x \in g_i^{-1}(U_i)$,  that is, $g_i(x) \in U_i$. Thus, $\Ob_G(x) \cap U_i \neq \emptyset$ for each $i=1,2,3,\ldots$. Since $\{U_i\}_{i\in \N}$ is a base for the topology of $X$, we have, orbit of $x$ is dense in $X$.
\end{proof}

\begin{theorem}
Let  $G$ be a semigroup of continuous self maps of a second countable Baire space $X$. If $X$ is topologically transitive then there exists $x\in X$ such that  orbit of $x$ is dense in $X$.
\end{theorem}

\begin{proof}
Since $X$ is topologically transitive, we have, for any two nonempty open subsets $U$ and $V$  of $X$, there exists a $g \in \widehat G$ such that $U \cap g^{-1}(V) \neq \emptyset$. Therefore, the set $\bigcup_{g \in \widehat G} g^{-1}(V)$ is dense in $X$. By Lemma \ref{bt}, there exists $D\subset X$, such that $D$ is residual in $X$ and $\Ob_G(x)$ is dense in $X$ for all $x \in D$. Since $X$ is a Baire  space and $D$ is  residual in $X$ and hence nonempty. Thus, there exists $x \in D$  such that   orbit of $x$ is dense in $X$.
\end{proof}

The following proposition shows that topological transitivity is preserved under conjugacy. More precisely,

\begin{proposition}
Let $(X,G)$ and $(Y,\ti{G})$ be two dynamical systems. If $\rho : X \to Y$ is a topological conjugacy, and $A \subset X$ is topologically transitive then so is $\rho(A)$.
\end{proposition}

\begin{proof}
Let $\ti U$ and $\ti V$ be two nonempty open subsets of $Y$ whose intersection with $\rho(A)$ is also nonempty. Since $\rho$ is a homeomorphism the sets $U=\rho ^{-1}(\ti U)$ and $V=\rho ^{-1}(\ti V)$ are open in $X$ and has nonempty intersection with A. Since $A$ is topologically transitive, there exists a $g\in \widehat G$ such that $g(U\cap A)\cap (V\cap A)\neq \emptyset$. 

Therefore, we have $\rho(g(U\cap A)\cap (V\cap A))\neq \emptyset$.  Since $\rho$  is a topological conjugacy, $\ti g\rho(U\cap A)\cap \rho(V\cap A)\neq \emptyset$, $\ti g \in \ti G$. Further, $\ti g(\ti U\cap \rho(A))\cap (\ti V\cap \rho(A))\neq \emptyset$. Thus, $\rho(A)$ is topologically transitive.
\end{proof}

In anticipation of defining the concept of chain transitive system, we need to invoke the definitions of a chain and chain recurrent point from \cite {kl1}.

Let $G$ be a semigroup on a metric space $(X,d)$. Let $a,b \in X$, $g\in G$ and $\e >0$ be  given. An $(\e,g)${\it-chain} from $a$ to $b$ means a finite sequence $(a=x_1,\ldots, x_{n+1}=b;g_1, \ldots, g_n)$, where for every $i, \ x_i\in X$ and $g_i \in \widehat G$ such that $d(g_i\circ g(x_i),x_{i+1})< \e$ for $i=1, \ldots ,n$.
%
Let $G$ be a semigroup on a metric space $(X,d)$. A pair of points $a,b\in X$ are called  {\it chain equivalent points} for $G$ if for every $\e>0$ and every $g\in G$ there exists an $(\e, g)$-chain from $a$ to $b$ and  an $(\e, g)$-chain from $b$ to $a$. The set of all chain equivalent points for $G$ is denoted by $CE(G)$.
%
Let $G$ be a semigroup on a metric space $(X,d)$. A point $x\in X$ is called a {\it chain recurrent point} for $G$ if for every $\e>0$ and every $g\in G$ there exists an $(\e, g)$-chain from $x$ to itself. The set of all chain recurrent points for $G$ is denoted by $CR(G)$.

\begin{definition}
Let $G$ be a semigroup on a metric space $(X,d)$. A nonempty subset $A$ of $X$ is said to be {\it chain transitive} if for each $a,b \in A$, $\e>0$ and $g\in G$, there exists an $(\e, g)$-chain from $a$ to $b$.
\end{definition}

\begin{remark} \label {rkct}
In the above definition, one can see, by interchanging the roles of  $a$ and $b$, any two points in a chain transitive set are chain equivalent. Also, letting $a=b$, we have, every element of a chain transitive set is chain recurrent.
\end{remark}

The following theorem shows that any topologically transitive subset is chain transitive.

\begin{theorem}
If $G$ is an abelian semigroup on a metric space $(X,d)$, then  any topologically transitive subset of $X$ with respect to $G$ is chain transitive.
\end{theorem}

\begin{proof}
Let $A$ be a topologically transitive subset of $X$ and $a,b \in A$. Let $g \in G$ and $\e >0$ be given. Since $g$ is continuous there exists a $\dl>0$ such that $d(a,y)<\dl$ implies $d(g(a),g(y))<\e$, $y\in X$.

Since $A$ is topologically transitive, for open sets $B(a,\dl)$ and $B(b,\e)$, by Remark \ref {rktt}, there exists $h\in \widehat G$ such that
$$g^2\circ h(B(a,\dl)\cap A) \cap (B(b,\e)\cap A) \neq \emptyset .$$
Since $G$ is abelian,
$$h \circ g^2(B(a,\dl)\cap A) \cap (B(b,\e)\cap A) \neq \emptyset .$$

Thus, there is a $c\in B(a,\dl)\cap A$  such that $hg^2(c)\in B(b,\e)\cap A$. That is,  $d(hg^2(c),b)<\e$ and $d(a,c)<\dl$ hence $d(g(a),g(c))<\e$. So that
$$(a,g(c),b;identity, h)$$
is an $(\e,g)$-chain from $a$ to $b$. Therefore, $A$ is chain transitive.
\end{proof}

\begin{theorem}
Let $(X,G)$ and $(Y,\ti{G})$ be two dynamical systems on the metric spaces $(X,d_X)$ and $(Y,d_Y)$. If a uniform homeomorphism $\rho : X \to Y$ is a topological conjugacy and $A\subset X$  is chain transitive then so is $\rho(A)$.
\end{theorem}

\begin{proof}
Let $x , \bar x \in A$.  Let $\e >0$ be given and some $\ti g \in \ti G$. We shall construct an $(\e,\ti g)$-chain from $y=\rho(x)$ to $\bar y=\rho(\bar x)$. Let $g \in G$ be such that $\rho \circ g=\ti g \circ \rho$.

Since $\rho : X \to Y$ is uniformly continuous, there exists a $\dl >0$ such that if $d_X(x_1,x_2)< \dl$ then $d_Y(\rho(x_1),\rho(x_2))< \e$, for $x_1,x_2 \in X$.  Also $x , \bar x \in A$ implies there is a $(\dl, g)$-chain $(x=x_1,\ldots , x_{n+1}=\bar x;g_1, \ldots ,g_n)$ from $x$ to $\bar x$. That is, for each $i=1, \ldots ,n$, we have,
$$d_X(g_i\circ g(x_i),x_{i+1})<\dl .$$

For  each $i=1, \ldots ,n+1$, let $y_i=\rho(x_i)$. Since $d_X(g_i\circ g(x_i),x_{i+1})<\dl $, we have,
$$ d_Y(\ti g_i\circ \ti g(y_i),y_{i+1})=d_Y(\rho \circ g_i\circ g(x_i),\rho(x_{i+1})) <\e.$$

Thus $(\rho(x)=y_1,\ldots , y_{n+1}=\rho(\bar x); \ti g_1, \ldots ,\ti g_n)$ is an $(\e,\ti g)$-chain  from $y=\rho(x)$ to $\bar y=\rho(\bar x)$.  Therefore $\rho(A) \subset Y$ is chain transitive.
\end{proof}

Next we show that a chain component is a maximal chain transitive set.

\begin{proposition}
The relation $R$ defined by: $aRb$ if and only if  $(a,b)\in CE(G)$, is an equivalence relation on $CR(G)$.
\end{proposition}

\begin{proof}
{\bf Reflexive:} Let $a \in CR(G)$. Then for every $\e>0$ and $g\in G$ there exists an $(\e, g)$-chain from $a$ to itself. Hence $aRa$

{\bf Symmetric:} Let $aRb$, that is,   $(a,b)\in CE(G)$. Then for every $\e>0$ and $g\in G$ there exists an $(\e, g)$-chain from $a$ to $b$ and  an $(\e, g)$-chain from $b$ to $a$. Thus $bRa$.

{\bf Transitive:} Let $aRb$ and $bRc$. Then for every $\e>0$ and $g\in G$ there exists an $(\e, g)$-chain from $a$ to $b$ and  an $(\e, g)$-chain from $b$ to $a$. Also $bRc$ implies for every $\e>0$ and $g\in G$ there exists an $(\e, g)$-chain from $c$ to $b$ and  an $(\e, g)$-chain from $b$ to $c$. By concatenating the two $(\e,g)$-chains from $a$ to $b$ and $b$ to $c$, we have, an $(\e,g)$-chain from $a$ to $c$. Again,  by concatenating the two $(\e,g)$-chains from $c$ to $b$ and $b$ to $a$, we have, an $(\e,g)$-chain from $c$ to $a$. Thus, we have, $aRc$.
\end{proof}

\begin{definition}
The equivalence relation $R$ defined by: $aRb$ if and only if  $(a,b)\in CE(G)$, is the chain equivalence relation for $G$ on $X$. An equivalence class of the chain equivalence relation for $G$ is called a {\it chain component} of $G$.
\end{definition}

\begin{proposition} \label {ct}
If $G$ is abelian then every chain component of $G$ is invariant under $G$.
\end{proposition}

\begin{proof}
Let $C$ be a chain component of $G$ on a metric space $(X,d)$. Let $x \in C$ and $g$ be any generator of $G$.  Let $\e >0$ be given and some $h \in G$. We shall construct  $(\e,h)$-chains from $g(x)$ to $x$ and $x$ to $g(x)$. It will follow that $g(x)\in C$.

Since $x \in CR(G)$ implies there is an $(\e, g\circ h)$-chain $(x=x_1,\ldots , x_{n+1}=x;h_1, \ldots ,h_n)$ from $x$ to itself. That is, for each $i=1, \ldots ,n$, we have,
$$d(h_i \circ g \circ h(x_i),x_{i+1})< \e.$$

In particular, since $G$ is abelian, we have,
\begin{equation} \notag
\begin{split}
d(h_1 \circ h \circ g(x_1),x_2)& =d(h_1 \circ g \circ h(x_1),x_2)\\
& \leq \e.
\end{split}
\end{equation}

Hence $(g(x),x_2,\ldots , x_n,x;h_1,h_2 \circ g, \ldots,h_{n-1}\circ g ,h_n\circ g)$ is  an $(\e,h)$-chain from $g(x)$ to $x$.

Further, since $g$ is continuous, there exists a $\dl \in (0, \e]$ such that if $d(x,y)< \dl$  then $ d(g(x),g(y)) < \e.$  Also $x \in CR(G)$ implies there is a $(\dl, g\circ h)$-chain $(x=x_1,\ldots , x_{n+1}=x;h_1, \ldots ,h_n)$ from $x$ to itself. That is, for each $i=1, \ldots ,n$, we have,
$$d(h_i \circ g \circ h(x_i),x_{i+1})<\dl \leq \e.$$

Also by the continuity of $g$, we have,
$$d(h_n \circ g \circ h(x_n),x)<\dl \ {\rm implies} \ d(g \circ h_n \circ g \circ h(x_n),g(x))<\e.$$

Hence $(x,x_2,\ldots , x_n,g(x);h_1\circ g,h_2 \circ g, \ldots,h_{n-1}\circ g ,g \circ h_n\circ g)$ is  an $(\e,h)$-chain from $x$ to $g(x)$.

Thus, $(x,g(x)) \in CE(G)$. Also,  by concatenating the two $(\e,g)$-chains from $g(x)$ to $x$ and $x$ to $g(x)$, we obtain, an $(\e,g)$-chain from $g(x)$ to itself. Hence $g(x)$ is a chain recurrent point and an element of $C$. Therefore, every chain component of $G$ is invariant under $G$, provided $G$ is abelian.
\end{proof}

\begin{corollary}
If $G$ is abelian then the chain recurrent set $CR(G)$ is invariant under $G$.
\end{corollary}

\begin{proof}
Since the collection of chain components partition the chain recurrent set. Also, by Proposition \ref {ct}, every chain component of $G$ is invariant,  provided $G$ is abelian. Therefore, the chain recurrent set  is invariant under $G$.
\end{proof}

\begin{theorem} 
Every chain component of $G$  is closed.
\end{theorem}

\begin{proof}
Let $C$ be a chain component of $G$ on a metric space $(X,d)$.
 Suppose that $a$ is a limit point of $C$ and $y\in C$. Let $\e >0$ be given and some $g \in G$. First we shall construct an $(\e,g)$-chain from $a$ to $y$.

Since $g$ is continuous, there exists a $\dl>0$  such that $d(x,a)<\dl \ {\rm implies} \ d(g(x),g(a))< \e$.

Take $\dl\rq{}=min\{\dl,\e/2\}$. Since $a$ is a limit point of $C$, there exists a $x \in C$ such that $d(x,a)<\dl\rq{}$. Then $d(g(x),g(a))< \e$. 

Since $x,y\in C$, there exists a $(\dl\rq{},g^2)$-chain $(x=x_1,\ldots, x_{n+1}=y;g_1, \ldots, g_n)$ from $x$ to $y$. So that $(a,g(x),x_2,\ldots, x_{n+1}=y;identity,g_1,g_2g ,\ldots, g_ng)$ is an $(\e,g)$-chain from $a$ to $y$.

Also, $x,y\in C$ implies there exists an $(\e/2,g)$-chain $(y=y_1,\ldots, y_{n+1}=x;h_1, \ldots, h_n)$ from $y$ to $x$. 
Since
\begin{equation} \notag
\begin{split}
d(h_ng(y_n),a) &<d(h_ng(y_n),x)+d(x,a)\\
&<\frac{\e}{2}+\frac{\e}{2}\\
&=\e,
\end{split}
\end{equation}
we have, $(y=y_1,\ldots, y_n,a;h_1, \ldots, h_n)$ is an $(\e,g)$-chain from $y$ to $a$.

By transitivity via concatinating these $(\e,g)$-chains, there is an $(\e,g)$-chain from $a$ to itself.  Hence $a$ is a chain recurrent point and an element of $C$. Therefore, every chain component of $G$ is closed.
\end{proof}

\begin{theorem}
A chain component of $G$ is a maximal chain transitive set.
\end{theorem}

\begin{proof}
By definition of a chain component, any two points in a chain components are chain equivalent. Hence, a chain component is a chain transitive set.

Furthermore, by Remark \ref {rkct}, every chain transitive set is contained in the chain recurrent set. Since the chain recurrent set has  partitioned  by disjoint equivalence classes of chain components. Therefore, any chain transitive set is contained in a unique chain component. Now, if $A$ and $B$ are two chain transitive sets of a semigroup, with  $A \subset B$ and $C$ is the unique chain component containing $A$. Then, $B \cap C \neq \emptyset$ and hence $C$ is the unique chain component containing $B$. Thus, a chain component of $G$ is a maximal chain transitive set.
\end{proof}

\section{Attractor}

This section consists of systematic  investigation to reproduce the concept of  attractors for semigroup and provide the alternative definition of a chain recurrent set. In \cite{conley}, Conley has introduced the notion of an attractor as a $\w$-limit set of a neighborhood of it. We shall define an attractor for a generalized  semigroup in analogous way.

In \cite {kl}, we have introduced the notion of $\w$-limit point for the generalized semigroup as follows:
%
A sequence of functions $(f_{n_k})$ in $G$ is said to be {\it{unbounded}} if  $n_k\to \ity$ as $k\to \ity$ and each $f_{n_k}$ consists of exactly $n_k$ iterates of  $g_{\al_{_0}}$, for fix $g_{\al_{_0}}\in \{g_{\al}\}_{\al}$, that is, $f_{n_k}=h_1\circ g_{\al_{_0}}\circ h_2 \circ g_{\al_{_0}} \circ h_3 \circ \ldots \circ h_{n_k}\circ g_{\al_{_0}} \circ h_{n_k+1}$, where each $h_i \in \widehat G= G\cup \{identity\}$ and the functions $h_i$ are independent of $g_{\al_{_0}}$.
%
A point $z\in X$ is called an {\it{$\w-$limit\ point}} for a point $x\in X$ if for some unbounded sequence $(f_{n_k})$ in $ G, \ f_{n_k}(x)\to z$ as $k \to \ity$. $\w(x)$ denotes the set of all  $\w$-limit points for $x$ .

 We introduce the notion of an  attractor for a  semigroup as follows:

\begin{definition} \label {tr}
Let $G$ be a semigroup on a metric space $(X,d)$. A nonempty open subset $U$ of $X$ is said to be a {\it trapping region} for $G$ if  there exists a $h\in  G$ such that  for the set
$$\widetilde U := \{fh(x): x\in U\ {\rm and \ } f \in \widehat G \}$$
we have,  $cl(\widetilde U) \subset U$.
\end{definition}

\begin{definition} \label {at}
Let $G$ be a semigroup on a metric space $(X,d)$. The {\it attractor} for $G$ determined by  a {\it trapping region} $U$ for $G$ is defined by
\begin{equation} \notag
\begin{split}
A:=&\{x\in X : {\rm for\ every\ open \ subset}\ V\ {\rm of}\  X\ {\rm containing}\  x,\ V \cap f_{n_k}(h(U))\ne \emptyset \\  
&\ {\rm for\  infinitely\  many}\  k \in \N, {\rm where}\ (f_{n_k})_k\ {\rm is\ some}\ {\rm unbounded\ sequence\  in\  } G \}
\end{split}
\end{equation}
where, $h\in G$ is as in the Definition \ref {tr}.
\end{definition}

\begin{definition}
Let $G$ be a semigroup on a metric space $(X,d)$. The {\it basin of attraction} of an attractor $A$ for $G$  is defined by
$$B(A):=\{x \in X : \w(x) \cap A \ne \emptyset \}.$$
\end{definition}

We notice the following characteristics of attractors and basin of attraction:

\begin{proposition}
The attractor $A$ determined by a trapping region $U$ is invariant under $G$.
\end{proposition}

\begin{proof}
Let $z \in A$ and $g_{\al}$ be a generator of $G$. Let $V$ be an open subset of $X$ containing $g_{\al}(z)$.  Then $g_{\al}^{-1}(V)$ is an open subset of $X$ containing $z$.

By  Definition \ref{at},  $g_{\al}^{-1}(V) \cap f_{n_k}(h(U))\ne \emptyset \ {\rm for\  infinitely\  many}\  k \in \N, {\rm where}\  (f_{n_k})_k\ {\rm is\ some}\\ {\rm unbounded\ sequence\  in\  } G$.  This gives $g_{\al}g_{\al}^{-1}(V) \cap g_{\al}f_{n_k}(h(U))\ne \emptyset$  and hence $V \cap g_{\al}f_{n_k}(h(U))\ne \emptyset$.

Since $(g_{\al}f_{n_k})$ is an unbounded sequence, it follows that $g_{\al}(z)\in A$.
\end{proof}

\begin{proposition}
The attractor $A$ determined by a trapping region $U$ is a closed set.
\end{proposition}

\begin{proof}
Let $y\in cl(A)$.  Then there is a sequence $(z_n)$ in $A$ converging to $y$. 

Let $V$ be an open subset of $X$ containing $y$. Since $(z_n)$ converges to $y$, there exists $N\in \N$ such that $z_n \in V$ for all $n \ge N$. Since $z_N \in A$ and V is an open set containing $z_N$, then by Definition \ref{at}, $ V \cap f_{n_k}(h(U))\ne \emptyset$  for  infinitely  many  $k \in \N$,  where  $(f_{n_k})_k$ is some unbounded sequence  in  $ G$. Thus, we have, $y \in A$.
\end{proof}

\begin{proposition} \label {atp1}
Let $A$ be the attractor determined by a trapping region $U$. If the phase space $X$ is compact then $A$ is contained in $U$.
\end{proposition}

\begin{proof}
Since $U$ is a trapping region there exists a $h\in  G$ such that  $cl(\widetilde U) \subset U$. Since $(X,d)$ is a compact metric space, $cl(\widetilde U)$ is compact. Thus, there exists $\eta >0$ such  that if $y \in cl(\widetilde U)$, $z \in X$, and $d(y,z)<\eta$, then $z\in U$.

Now, if $x \in A$ and $\e \in(0, \eta]$ then $B_{\e}(x)\cap f_{n_k}(h(U)) \ne \emptyset$  for  infinitely  many  $k \in \N$,  where  $(f_{n_k})_k$ is some unbounded sequence  in  $ G$ and $B_{\e}(x)$ is an open ball with centre $x$ and radius $\e$.  Then
$$ f_{n_k}(h(U)) \subset \widetilde U \subset cl(\widetilde U).$$

Therefore, $B_{\e}(x)\cap cl(\widetilde U) \ne \emptyset$. Let $y \in B_{\e}(x)\cap cl(\widetilde U)$. Since
$$d(x,y)<\e \le \eta ,$$
we have, $x\in U$.
\end{proof}

\begin{proposition} \label {atp2}
Let $A$ be the attractor determined by a trapping region $U$. If the phase space $X$ is compact  then $B(A)$ contains $U$. 
\end{proposition}

\begin{proof}
Let $x\in U$ and $(f_{n_k})$ be an unbounded sequence.

Since $X$ is compact there is a $z\in X$ and a subsequence $(n_{k_l})$ such that $f_{n_{k_l}}h(x) \to z$, where, $h\in  G$ is such that  $cl(\widetilde U) \subset U$. Let $V$ be an open subset of $X$ containing $z$. Then there exists  a $N \in \N$ such that
$$V \cap f_{n_{k_l}}h(x) \ne \emptyset, \quad {\rm for\ all} \ l\ge N.$$
Thus,
$$V \cap f_{n_{k_l}}h(U) \ne \emptyset, \quad {\rm for\ all} \ l\ge N.$$

Therefore, $z \in A \cap \w(x)$ and hence $U \subset B(A)$.
\end{proof}

\begin{theorem}  \label{atth}
Let $(X,d)$ be a compact metric space and $G$ be an abelian semigroup. The chain recurrent set of $G$ is the complement of the union of sets $B(A) \setminus A$ as $A$ varies over the collection of attractors  of $G$:
$$X \setminus CR(G)=\bigcup_A[B(A) \setminus A].$$
\end{theorem}

\begin{proof}
Let $A$ be an attractor of $G$ and $p\in B(A) \setminus A$. Let $U$ be a trapping region which determines $A$, and $h\in  G$ is such that  $cl(\widetilde U) \subset U$. As in Proposition \ref {atp1}, there exists $\eta >0$ such  that if $y \in cl(\widetilde U)$, $z \in X$, and $d(y,z)<\eta$, then $z\in U$.

For $p\in B(A)$, we have, $\w (p) \cap A \ne \emptyset$.  Let  $z\in \w (p) \cap A \subset U$. Therefore, there exists an unbounded sequence $(f_{n_k})$ in $G$ and a $N \in \N$ such that
$$f_{n_k}(p) \to z$$
and
$$\qquad f_{n_k}(p)\in  U,\quad {\rm for\ all} \ k\ge N.$$

Now, if $p \in CR(G)$, then for $k \ge N$ and  $\e \in(0, \eta]$ there exists an $(\e,f_{n_k}h)$-chain $(p=x_1, \ldots , x_{l_k+1}=p;h_1, \ldots ,h_{l_k})$ from $p$ to itself. Since $G$ is abelian and $ f_{n_k}(p)\in  U$, we have,
$$h_1f_{n_k}h(x_1) \in h_1h(U) \subset  \widetilde U \subset cl(\widetilde U).$$
Since 
$$d(h_1f_{n_k}h(x_1),x_2)<\e \le \eta ,$$
we have, $x_2 \in  U$. Similarly,
$$h_2f_{n_k}h(x_2) \in  \widetilde U \subset cl(\widetilde U)$$
and $d(h_2f_{n_k}h(x_2),x_3)<\e \le \eta $ gives $x_3 \in U$. Following the previous arguments, we have $x_i \in U$ for every $i=1, \ldots ,{l_k}$. Since 
$$d(h_{l_k}f_{n_k}h(x_n),p)<\e,$$
we have,
$$B_{\e}(p)\cap h_{l_k}f_{n_k}h((U)) \ne \emptyset.$$

Now, the sequence of functions defined by
\begin{equation}\notag
f_{m_k}=
\begin{cases}
h_{l_k}f_{n_k}, & k \ge N\\
f_{n_k}, & k < N
\end{cases}
\end{equation}
is an unbounded sequence in $G$. Also for any open subset $V$ of $X$ containing $p$, there is $\e \in(0, \eta]$ such that $B_{\e}(p) \subset V$ and hence
$$V \cap f_{m_k}(h(U)) \supset B_{\e}(p)\cap f_{m_k}(h(U)) \ne \emptyset ,$$
for all $k \ge N$. Thus $p \in A$, which is a contradiction. 

Therefore, if $p\in B(A) \setminus A$ then $p \notin CR(G)$.

Conversely, let $p \notin CR(G)$. Then there exists $\e >0$ and $h_1 \in  G$ such that there is no $(\e,h_1)$-chain from $p$ to itself. Consider the set defined by
$$U:=\{x\in X:{\rm there\ is\ an}\ (\e,h_1){\rm -chain\ from\ } p\ {\rm to\ }x \}.$$
Then $p \notin U$ and  
$$\widetilde U := \{fh_1(x): x\in U\ {\rm and \ } f \in \widehat G \}.$$

Also $U$ is an open subset of $X$. Let $x\in U$, there exists an $(\e,h_1)$-chain $(p=x_1, \ldots , x_{n+1}=x;f_1, \ldots ,f_n)$ from $p$ to $x$. Now, $x\in B_{\e}(f_nh_1(x_n)) \subset U$. If $y \in B_{\e}(f_nh_1(x_n))$ then $(p=x_1, \ldots , x_{n+1}=y;f_1, \ldots ,f_n)$ from $p$ to $x$ is  an $(\e,h_1)$-chain .

Let $y \in cl(\widetilde U)$ . Then for some $x \in U$ and $f \in \widehat G$, we have, $f h_1(x) \in B_{\e}(y)$. Thus $(x,y;f)$ is an  $(\e ,h_1)$-chain from $x$ to $y$. By transitivity, we can produce an $(\e,h_1)$-chain from $p$ to $y$.
 Hence $cl(\widetilde U)) \subset U$.

Therefore, $U$ is a trapping region for $G$.

Let $A$ be the attractor determined by $U$. Since $p \notin U$ and $A \subset U$, it follows that $p \notin A$.

Moreover, $(p,h_1(p); identity)$ is an $(\e,h_1)$-chain from $p$ to $h_1(p)$. Hence, $h_1(p)\in U$.  By Proposition \ref {atp2}, we have, $\w(h_1(p)) \cap A \ne \emptyset$. Thus, $\w(h_1(p)) \subset \w(p)$ implies $p \in B(A)$. Therefore, $p \in B(A) \setminus A$.
\end{proof}

Thus, Theorem \ref{atth} generalizes the result obtained by  Hurley in \cite {mh} following Conley\rq{}s characterization in \cite {conley}.

The following example will illustrate the application of the above theorem.
\begin{example}
Consider the family of  polynomial mappings on closed unit disc in complex plane, that is $X=\{z\in \C : |z|\le 1\}$, for each $n\in \N$, given by
$$g_n : X \to X$$
$$g_n(z)=z^n.$$

The semigroup $G=\{ g_n: n \in \N\setminus\{1\}\}$, which is closed under composition, consists of non-identity continuous self maps of the closed unit disc in complex plane $\C$. Now, by Fundamental Theorem of Arithmetic, we can see that the semigroup $G$  has set of generators given  by $\{g_p: p{\rm\  is\ a\ prime\ natural \ number}\}$.

We shall identify all the trapping regions and attractors for the dynamics of $G$ on the phase space $X$. Evidently, the empty set and the phase space itself are trapping regions for $G$. Moreover, in this case, the corresponding attractor and basin of attraction coincide with the trapping regions respectively. Thus, we have, $B(A) \setminus A=\emptyset$ for each attractor.

 For each $r\in (0,1)$, the set $U_r=\{z\in \C :|z| < r\}$ is a trapping region for $G$. Since $z\in U_r$ implies $|g(z)| \leq |z|^2 <r$ for every $g\in G$. In particular, for $\widetilde U_r=\{gg_{_2}(z): z\in U_r {\rm \ and\ } g \in \widehat G\}$, we have, $cl(\widetilde U_r) \subset U_r$. Here, the attractor $A_r$ is $\{0\}$ for each $r$ and the basin of attraction $B(A_r)$ is the set $\{z\in \C : |z|<1\}$ for each $r$. Therefore, in this case, $B(A_r) \setminus A_r=\{z\in \C : 0<|z|<1\}$.

Since the action of elements of the semigroup results on the set $\{z\in \C : |z|<1\}$  in radial contraction. Therefore, the attractors and basin of attraction for any proper non empty trapping region correspond to the above pair.

By Theorem \ref{atth}, we have,
\begin{equation} \notag
\begin{split}
X \setminus CR(G) &= \bigcup_A[B(A) \setminus A] \\
&=\emptyset \cup \{z\in \C : 0<|z|<1\} \\
&= \{z \in \C : |z|\neq 0, |z|\neq 1\}.
\end{split}
\end{equation}

Consequently, we have, $CR(G)=\{z \in \C : |z|= 0, |z|= 1\}$.

Next we shall find all the chain recurrent points for $G$  and show that $CR(G)=\{z \in \C : |z|= 0, |z|= 1\}$.

Since 0 is a fixed point of $G$, hence a chain recurrent point for $G$.

Let $z\in \{z\in \C : 0<|z|<1\}$, we shall show that $z$ is not a chain recurrent point for $G$. Let $m,M \in \N$ be sufficiently large, so that $1/m$ is small enough such that $|g_{_M}(z)-z|>\frac{1}{m}$.  Then there is no $(\frac{1}{m},g_{_M})$-chain from $z$ to itself. Thus, $z$ is not a chain recurrent point for $G$.

Let $z_0 \in \C$ be such that $|z_0|=1$, that is, $z_0=Exp(\io \pi \ta_0)$.  Let $\e >0$ and $n_0\in \N \setminus \{1\}$. We shall construct an $(\e, g{_{n_0}})$-chain from $z_0$ to itself. 

Since rationals are dense in reals, we can pick a $z_1=Exp(\io \pi p_1/q_1) \in B_{\e}( g{_{n_0}}(z_0))=\{z:|g{_{n_0}}(z_0)-z|<\e\}$, for some $p_1,q_1 \in \N $. Then, 
\begin{equation} \notag
\begin{split}
g_{2q_1}g{_{n_0}}(z_1) &=Exp(\io \pi p_1 n_0 2q_1/q_1)\\
&=1.
\end{split}
\end{equation}

Let $p,q \in \N$ be such that $w=Exp(\io \pi p/q) \in B_{\e}(z_0)$. There is a $N \in \N$ sufficiently large such that $z_2=Exp(\io \pi p/(qn_0N)) \in B_{\e}(1)$. Then $g_{_N}g_{n_0}(z_2)=w$ and $|w-z_0|< \e$.

Therefore, $(z_0,z_1,z_2,z_0;identity,g_{2q_1},g_{_N} )$ is an $(\e, g{_{n_0}})$-chain from $z_0$ to itself. Thus, $z_0$ is a chain recurrent point for $G$ and, we have, $CR(G)=\{z \in \C : |z|= 0, |z|= 1\}$.
\end{example}

{\bf Acknowledgment.}
I am thankful to my thesis adviser Sanjay Kumar for fruitful discussions.

\end{document}